\documentclass[11pt,reqno]{amsart}
\usepackage{amsmath, amsthm, amscd, mathrsfs, amsfonts, amssymb,graphicx, color}
\newtheorem{thm}{Theorem}[section]

\newtheorem{lem}[thm]{Lemma}
\newtheorem{prop}[thm]{Proposition}
\newtheorem{defn}[thm]{Definition}
\newtheorem{rem}[thm]{\bf{Remark}}

\numberwithin{equation}{section}
\newtheorem{exm}[thm]{Example}

\begin{document}
\begin{center}
\Large{{\bf{Impulsive Fractional Dynamic Equation with Non-local Initial Condition on Time Scales}}}
\vspace{.5cm}

Bikash Gogoi$^{a},$  Bipan Hazarika$^{b}$\footnote{Corresponding author} and Utpal Kumar Saha$^{c}$ 

\vspace{.2cm}
\footnotesize $^{a,c}$Department of Basic and Applied Science, National Institute of Technology Arunachal Pradesh, Jote, Arunachal Pradesh 791113, India\\
\footnotesize  $^{b}$Department of Mathematics, Gauhati University, Guwahati 781014, Assam, India

\vspace{.2cm}
Email: $^{a}$bikash.phd20@nitap.ac.in;  $^{b}$uksahanitap@gmail.com;\\ $^{c}$bh\_rgu@yahoo.co.in, bh\_gu@gauhati.ac.in
\end{center}
\title{}
\author{}
\thanks{{\today}}
\maketitle
\begin{abstract}
	In this manuscript we investigate the existence and uniqueness of an impulsive fractional dynamic equation on time scales involving non-local initial condition with help of Caputo nabla derivative. The existency is based on the Scheafer's fixed point theorem along with the Arzela-Ascoli theorem and Banach contraction theorem. The comparison of the Caputo nabla derivative and Riemann-Liouvile nabla derivative of fractional order are also discussed in the context of time scale.  
\end{abstract}
\maketitle

\section{Introduction}
In our real world, some situation may arise which cannot be modeled in entirely continuous phenomena or in entirely discrete phenomena, in such situation we need a common domain which justify both the conditions. On the basis of unification of these conditions Stefan Hilger introduced a common domain called time scale $\mathbb{T}$ which unifies both continuous and discrete calculus \cite{HA, HD}. Dynamic equations on time scale were introduced to solve this kind of model which is a combination of both the differential and difference equation. Many researchers worked on dynamic equation in linear and non linear, involving local initial and boundary conditions. Many authors discussed the dynamic equation using the tool of fractional calculus due to the accuracy and advantage in the physical interpretation. For detailed study of fractional dynamic equations by Caputo, Rieamann-Liouville, Caputo-Hadamard and many others, readers can go through the manuscripts \cite{RM, AI,NB,MA,MD,KD,BBU,CJ,MR,SKA,BM,S,JL,ZY} and the references therein.\\
\indent In real world situation, we have seen some equations where the systems are allowed to undergo some sudden perturbation, whose duration can be negligible in comparison with the duration of the process. In this case the solution of these equations may have jump discontinuities at time $\theta_{1} < \theta_{2} < \theta_{3} < \cdot\cdot\cdot,$ given in the form $p(\theta_{k}^{+}) - p(\theta_{k}^{-}) = \mathscr{I}_{k}(\theta_{k}, p(\theta_{k}^{-}))$. The dynamic equations having jump discontinuities for their solutions are called impulsive dynamic equations. The theory have interesting applications in the branch of mathematical modeling of different types of real world situation which require sudden changes at a particular time of their evolution, for example natural disaster, particular diseases, etc. Work related to impulsive dynamic equations can be seen in the manuscripts \cite{K,ZG,VM,VK,IG,SC} and the reference therein. In recent time several researchers and authors have shown their attention in the topic of impulsive dynamic equations on time scales, However a few number of works have been seen in the impulsive dynamic equation by using the fractional calculus on time scales with non local initial condition.\\
\indent In the manuscript \cite{SMB,SC}, the authors discussed the impulsive dynamic equation in terms of non local initial condition, whereas in \cite{VM,VK}, the authors investigated the fractional dynamic equation with instantaneous and non-instantaneous impulses with local initial condition by using the tools of delta (Hilger) derivative. A dynamic model with non local initial conditions are applicable in all branches of science and engineering, due to the advantage of using non local initial conditions such as, the measurement at more places which can be incorporated to get the better model. For detailed study of the advantages of non-local initial conditions, one can see \cite{RK} and the references therein.  \\
\indent So motivated by aforementioned work we claim that it is worthwhile to study the impulsive fractional dynamic equation with non-local initial condition of the type:
\begin{equation}\label{eq1}
\begin{cases}
^{C}\mathscr{D}^{w}p(\theta) = \mathcal{L}\big(\theta, p(\theta), ^{C}\mathscr{D}^{w}p(\theta)\big), & \theta \in \mathcal{I}_{T}, \theta \neq \theta_{k}\\
p(\theta_{k}^{+}) - p(\theta_{k}^{-}) = \mathscr{I}_{k}\big(\theta_{k}, p(\theta_{k}^{-}), & k = 1, 2, 3, \cdot\cdot\cdot, n\\
p(0) = \phi(p),
\end{cases}
\end{equation}
where  $k \in \mathbb{N} \cup \{{0}\},$ and $\mathcal{I}_{T} = [0, T] \cap \mathbb{T},$ for $T \in \mathbb{T}$ denote the time scale interval. $\mathcal{L}: \mathcal{I}_{T} \times \mathbb{R} \times \mathbb{R} \to \mathbb{R}$ is a ld continuous function, and $^{C}\mathscr{D}^{w}$ denote the Caputo nabla derivative of order $w \in (0, 1)$ which is discussed in. We assume that $0 < \theta_{0} < \theta_{1} < \theta_{2} < \theta_{3}< . . .< \theta_{n} < \theta_{n + 1} = {T}$ represents the impulse at a certain moment, and the term $p(\theta_{k}^{+}) = \lim_{d \to 0}p(\theta + d)$ and $p(\theta_{k}^{-}) = \lim_{d \to 0}p(\theta - d)$ represents the right and left limits of the function p at $\theta = \theta_{k}$ in the context of time scales. $\mathscr{I}_{k}$ is a continuous real valued function on $\mathbb{R}$ for each $k =1, 2, 3, \cdot \cdot \cdot, m$ and $\mathscr{I}_{k}(\theta_{k}, p(\theta_{k}^{-}))$ are the impulses acted on the time scale interval $\mathcal{I}_{T}$ which will be specified later.\\
The manuscript is organized as follows. In Section 2, we have presented some auxiliary results related to fractional dynamic equation on time scale, which will be required to show our main findings. In Section 3, we compare the Riemann-Liouville and Caputo nabla derivative in the context of time scale. In Section 4, we have given the existence and uniqueness theorem of a impulsive fractional dynamic equation with non local initial condition. In Section 5, we provided an example, which makes the manuscript easier to understand. Finally, conclusion of the paper is presented in Section 6.
\section{Auxilary results}
\begin{defn}\label{d1}\cite{BBU}
A function $\alpha : \mathbb{T} \to \mathbb{R},$ defined by $\alpha(\theta) = \{ \zeta \in \mathbb{T}: \zeta < \theta\}$ is said to be backward jump operator. Any $\theta \in \mathbb{T}$ is said to be left dense if $\alpha(\theta) = \theta$ and if $\alpha(\theta) = \theta - 1,$ then $\theta$ is said to be a left scattered point on $\mathbb{T}.$\\
\textbf{Remark}: If $\mathbb{T}$ has a minimum right scattered point say $y,$ then set $\mathbb{T}_{\mathcal{V}} = \mathbb{T} \setminus \{y\},$ else $\mathbb{T}_{\mathcal{V}} = \mathbb{T}.$
\end{defn}
\begin{defn}\label{d2}\cite{SC}
A function $x: \mathbb{T} \times \mathbb{R} \times \mathbb{R} \to \mathbb{R}$ is said to be left dense continuous function, if $x\big(\cdot, u, v \big)$ is left dense continuous on $\mathbb{T}$ for each ordered pair $( \theta, \zeta) \in \mathbb{R} \times \mathbb{R},$ and $x(\theta, \cdot, \cdot)$ is continuous on $\mathbb{R} \times \mathbb{R}$ for fixed point $\theta \in \mathbb{T}.$\\
\indent The set  of all left dense continuous function from $\mathbb{T} \times \mathbb{R} \times \mathbb{R}$ to $\mathbb{R}$ is denoted by $\mathcal{C}\big(\mathbb{T}, \mathbb{R}\big).$ 
\end{defn}
\begin{defn}\label{d3}\cite{VK}
Consider a function $g : \mathbb{T} \to \mathbb{R}$. Let $\mathcal{G}$ be a function such that $\mathcal{G}_{\nabla}(\theta) = g(\theta),$ for each $\theta \in \mathbb{T_{\mathcal{V}}},$ then the nabla integral is presented by 
\begin{align*}
\int_{a}^{\theta}g(x)\nabla x = \mathcal{G}(x) - \mathcal{G}(a).
\end{align*}
\end{defn}
\begin{prop}\label{p01}\cite{NB}
Let $g$ be an increasing continuous function on the time scale interval $[0, T] \cap \mathbb{T}.$ If $\mathcal{G}$ is the extension of $g$ to the real line interval $[0, T], T\in \mathbb{R},$ we get
$$\mathcal{G}(\theta) = \begin{cases}
g(\theta), & \mbox{if~} \theta \in \mathbb{T},\\
g(\zeta), & \mbox{if~} \theta \in (\theta, \alpha(\theta)) \notin \mathbb{T},
\end{cases}$$ then \begin{equation}\label{eq001}
\int_{r}^{u}g(\theta)\nabla \theta \le \int_{r}^{u}\mathcal{G}(\theta) d\theta,
\end{equation} for $r, u \in [0, T] \cap \mathbb{T},$ such that $r < u.$
\end{prop}
\begin{defn}\label{d04}\cite[Higher order nabla derivative]{BBU} Assume  $\mathcal{H}: \mathbb{T_{\mathcal{V}}} \to \mathbb{R}$ is a ld continuous function on a time scale $\mathbb{T}.$ The second order nabla derivative $\mathcal{H}_{\nabla \nabla} = \mathcal{H}_{\nabla}^{(2)}$ can be define,  provided  $\mathcal{H}_{\nabla}$ is differentiable on $\mathbb{T_{\mathcal{V}}}^{(2)} = \mathbb{T_\mathcal{VV}}$ with derivative $\mathcal{H}_{\nabla}^{(2)} = (\mathcal{H}_{\nabla})_{\nabla} : \mathbb{T_{\mathcal{V}}}^{(2)} \to \mathbb{R}$. Similarly, proceeding upto $n^{th}$ order we get $\mathcal{H}_{\nabla}^{(n)} : \mathbb{T_{\mathcal{V}}}^{n} \to \mathbb{R},$ it is attained by cut out  $n$ right scattered left end points from $\mathbb{T}.$
\end{defn}
\begin{defn}\label{d5}\cite{BBU}
Let $\mathcal{H}: \mathbb{T}_{\mathcal{V}^{n}} \to \mathbb{R}$ be a ld continuous function, such that $\mathcal{H}_{\nabla}^{(n)}$ derivative exists. Then the Caputo nabla derivative is defined by 
\begin{align*}
^{C}\mathscr{D}^{w}_{a}\mathcal{H}(\theta) = \frac{1}{\Gamma(n - w)}\int_{a}^{\theta}(\theta - \alpha(\zeta))^{n - w - 1}\mathcal{H}_{\nabla}^{(n)}(\zeta)\nabla \zeta,
\end{align*} 
for $n = [w] + 1.$ If $w \in (0, 1),$ then 
\begin{align*}
^{C}\mathscr{D}^{w}_{a}\mathcal{H}(\theta) = \frac{1}{\Gamma(1 - w)}\int_{a}^{\theta}(\theta - \alpha(\zeta))^{- w}\mathcal{H}_{\nabla}\nabla \zeta,
\end{align*}
where $[.]$ is used to denote the greatest integer function.
\end{defn}
\begin{defn}\label{d6}\cite{BBU}
Let $\mathcal{H}$ be any ld continuous function define on the set $\mathbb{T_{\mathcal{V}}}.$ Then the Riemann-Liouville fractional nabla derivative of order $w \in (0, 1)$ is denoted by 
\begin{align*}
\mathscr{D}^{w}_{\theta_{0}}x(t) = \frac{1}{\Gamma(1 - w)}\Big(\int_{\theta_{0}}^{\theta}\big(\theta - \alpha(\zeta)\big)^{- w}x(\zeta)\nabla \zeta\Big)^{\nabla}.
\end{align*}
\end{defn}
\begin{defn}\label{d7}{\cite[Definition 13]{NB}}
Let $\mathcal{H}: \mathcal{I_{T}} \to \mathbb{R}$ be an integrable function. Then the nabla fractional integral of $\mathcal{H}$ is given by 
\begin{align*}
\mathcal{D}_{\theta_{0}}^{- w}\mathcal{H}(\theta) = \mathcal{J}^{w}_{\theta_{0}}\mathcal{H}(\theta) = \frac{1}{\Gamma(w)}\int_{\theta_{0}}^{\theta}(\theta - x)^{w - 1}\mathcal{H}(x)\nabla x.
\end{align*}
\indent The Rieamm-Liouville fractional integral always satisfies the condition \begin{align*}
\mathcal{J}^{w}_{\theta_{0}}\mathcal{J}^{u}_{\theta_{0}}\mathcal{H}(\theta) = \mathcal{J}^{w + u}_{\theta_{0}}\mathcal{H}(\theta). 
\end{align*}
\end{defn}
\begin{lem}\label{l1}{\cite[Definition 13]{NB}}
If $p(\theta)$ is a ld continuous function, then $$\begin{cases}\mathcal{D}^{u}\mathcal{J}^{w}p(\theta) = p(\theta)\\
\mathcal{D}^{u}\mathcal{J}^{w}p(\theta) = \mathcal{J}^{w - u}p(\theta).
\end{cases}$$ 
\end{lem}
\begin{defn}\label{d8}\cite{VK}
Let $\mathscr{D} \subset \mathcal{C}(\mathbb{T}, \mathbb{R})$ be a set. $\mathscr{D}$ is a relatively compact, if it is bounded and equicontinuous simultaneously.
\end{defn}
\begin{defn}\label{d9}\cite{SC}
A mapping $\mathcal{H} : A \to B$ is completely continuous, if for a bounded subset $\mathcal{B} \subseteq A,$ $\mathcal{H}(\mathcal{B})$ is relatively compact in $A.$
\end{defn}
\begin{defn}\label{d10}{\cite[Schaefer's fixed point theorem]{JYJ}}
Let $\mathscr{A}$ be a Banach space and $\mathcal{H}: \mathscr{A} \to \mathscr{A}$ is a completely continuous mapping, if the set $\Psi = \{p \in \mathscr{A}: p = \lambda \mathcal{H}(p), 0 < \lambda < 1\}$ is bounded, then the operator $\mathcal{H}$ has a fixed point in $\mathscr{A}.$
\end{defn}
\begin{defn}\label{d11}\cite{SC}
If $\mathcal{L}: \mathcal{I_{T}} \times \mathbb{R} \times \mathbb{R}$ is a ld continuous function, then for $w \in (0, 1),$ a function $p$ is solution of 
\begin{align*}
^{C}D^{w}p(\theta) = \mathcal{L}\big(\theta, p(\theta), ^{C}D^{w}p(\theta)\big), p(\theta)|_{\theta = 0} = \phi(p)
\end{align*} if and only if $p$ is the solution of the integral equation
\begin{align}\label{eq01}
p(\theta) = \phi(p) + \frac{1}{\Gamma(w)}\int_{0}^{\theta}\big(\theta - \alpha(x))^{w - 1}\mathcal{L}(x, p(x), ^{C}D^{w}p(x)\big)\nabla x.
\end{align}
\end{defn}
\section{Comparison of Riemann-Liouville and Caputo nabla fractional derivative}
\begin{prop}\label{p1}
For any $w \in \mathbb{R},$ let $m - 1 < w < m, m \in \mathbb{N}$ such that $^{C}\mathscr{D}^{w}_{\theta_{0}}\mathcal{G}(\theta)$  exists in the time scale $\mathbb{T},$ then
\begin{center} 
$^{C}\mathscr{D}^{w}_{\theta_{0}}\mathcal{G}(\theta) = \mathcal{J}^{m - w}\mathcal{G}_{\nabla}^{(m)}(\theta).$
\end{center}
\begin{proof}
From the Definition \ref{d5} and Definition \ref{d7} we get
$$^{C}\mathscr{D}^{w}_{\theta_{0}}\mathcal{G}(\theta) = \mathcal{J}^{m - w}\mathcal{G}_{\nabla}^{(m)}(\theta).$$
\end{proof}
\end{prop}
\begin{thm}\label{th1}
For any $\theta \in \mathbb{T}_{\mathcal{V}^{n}},$ the nabla Caputo derivative and nabla Riemann-Liouville derivative of the fractional order $w,$ where $m = [w] + 1$ satisfies the following relation:
\begin{align*} 
^{C}\mathscr{D}^{w}\mathcal{G}(\theta) = \mathcal{D}^{w}\big(\mathcal{G}(\theta) - \sum_{v = 0}^{m - 1}\frac{(\theta - \rho)^{v}}{\Gamma(w + 1)}\mathcal{G}(\rho)\big).
\end{align*}
\end{thm}
The proof of this theorem is based on the Taylor's theorem defined in {\cite[Theorem 10]{JL}}.
\begin{proof}
Let us consider a ld continuous function $\mathcal{G}$ which is n times nabla differentiable, then for any fixed $\rho \in \mathbb{T}$ and $m \in \mathbb{N}\cup \{0\}, m < n$ we get
\begin{align}
\mathcal{G}(\theta) =&\sum_{v = 0}^{m - 1}\frac{(\theta - \rho)^{v}}{\Gamma(v + 1)}\mathcal{G}_{\nabla}^{(v)}(\rho) + \frac{1}{\Gamma(m)}\int_{\rho}^{\theta}(\theta - \alpha(\zeta)\mathcal{G}_{\nabla}^{(m)}(\zeta)\nabla \zeta\nonumber\\
= &\sum_{v = 0}^{m - 1}\frac{(\theta - \rho)^{v}}{\Gamma(v + 1)}\mathcal{G}_{\nabla}^{(v)}(\rho) + \mathcal{J}_{\rho}^{m}\mathcal{G}^{(m)}_{\nabla}(\theta)\label{eq2}.
\end{align}
Now taking the Riemann-Liouville derivative $\mathcal{D}^{w}_{\rho}$ of order $w$ in both side of Equation (\ref{eq2}). Now by using Lemma \ref{l1} and Proposition \ref{p1}, we get
\begin{align}
\mathcal{D}_{\rho}^{w}\mathcal{G}(\theta) = &\mathcal{D}_{\rho}^{w}\sum_{v = 0}^{m - 1}\frac{(\theta - \rho)^{v}}{\Gamma(v + 1)}\mathcal{G}_{\nabla}^{(v)}(\rho) + \mathcal{D}_{\rho}^{w}\mathcal{J}_{\rho}^{(m)}\mathcal{G}_{\nabla}^{(n)}(\theta)\nonumber\\
= & \mathcal{D}_{\rho}^{w}\sum_{v = 0}^{n - 1}\frac{(\theta - \rho)^{v}}{\Gamma(v + 1)}\mathcal{G}_{\nabla}^{(v)}(\rho) + \mathcal{J}_{\rho}^{m - w}\mathcal{G}_{\nabla}^{(n)}(\theta)\nonumber\\
= & \mathcal{D}_{\rho}^{w}\sum_{v = 0}^{m - 1}\frac{(\theta - \rho)^{v}}{\Gamma(v + 1)}\mathcal{G}_{\nabla}^{(v)}(\rho) + ^{C}\mathscr{D}_{\rho}^{w}\mathcal{G}(\theta)\label{eq3}.
\end{align}
From the above we get
\begin{align}
^{C}\mathscr{D}_{\rho}^{w}\mathcal{G}(\theta) = &\mathcal{D}_{\rho}^{w}\mathcal{G}(\theta) -   \mathcal{D}_{\rho}^{w}\sum_{v = 0}^{m - 1}\frac{(\theta - \rho)^{v}}{\Gamma(v + 1)}\mathcal{G}_{\nabla}^{(v)}(\rho)\nonumber\\
=&\mathcal{D}_{\rho}^{w}\big(\mathcal{G}(\theta) - \sum_{v = 0}^{m - 1}\frac{(\theta - \rho)^{v}}{\Gamma(v + 1)}\mathcal{G}_{\nabla}^{(v)}(\rho)\big).\label{eq4}
\end{align}
\end{proof}
\begin{prop}\label{p2}
If $ w\in (0, 1),$ then from the Theorem \ref{th1} and applying the Equation (\ref{eq4}) we obtain $m = 1,$ hence
\begin{align*}
^{C}\mathscr{D}_{\rho}^{w}(\theta) = \mathcal{D}_{\rho}^{w}\big(\mathcal{G}(\theta) - \mathcal{G}(\rho)\big).
\end{align*}
Case 1:  If the initial condition $\mathcal{G}(\rho) \to 0,$ as $\rho \to 0,$ then we obtain
\begin{equation}\label{eq5}
^{C}\mathscr{D}^{w}\mathcal(\theta) = \mathcal{D}^{w}\mathcal{G}(\theta).
\end{equation}
Thus the Caputo nbala derivative of order $w \in (0, 1)$ coincide with the Riemann-Liouville nabla derivative.\\
Case 2: If $ w \in \mathbb{N},$ then from the Equation (\ref{eq2}) and using the Lemma \ref{l1}, we obtain
\begin{align*}
^{C}\mathscr{D}^{m}\mathcal{G}(\theta) = &\mathcal{D}^{m}\big(\mathcal{G}(\theta) - \sum_{v = 0}^{m - 1}\frac{(\theta - \rho)^{v}}{\Gamma(v + 1)}\mathcal{G}(\rho)\big)\\
= &\mathcal{D}^{m}\mathcal{J}^{m}\mathcal{G}_{\nabla}^{(m)}(\theta)\\
= &\mathcal{G}_{\nabla}^{(m)}(\theta).
\end{align*}
\indent Thus the Caputo nabla derivative is coincide with the nabla derivative.
\end{prop}
\begin{rem}\label{r1}
When the initial condition is given in terms of real order in any types of dynamic equation involving impulses, the application of Caputo nabla derivative is mostly preferable over the Riemann-Liouville derivative due to its physical interpretation see \cite{NB,IF}. However in terms of integral order the accuracy of nabla Caputo derivative and Riemann-Liouville nabla derivative are almost same.
\end{rem}
\section{Existence and uniqueness of Impulsive fractional dynamic equation}
At first, we compare the dynamic Equation (\ref{eq1}) with a model of population dynamics with a stop start phenomena, where $p(\theta)$ is a population of a particular species of insect at a time $\theta.$ If we include the toxic effect on that particular species, and we noticed the change of population which is present by the Caputo derivative operator $^{C}\mathscr{D}^{w}p(\theta)$ of the fractional order w, with respect to time $\theta$ on the interval $\mathcal{I_{T}} = [0, T]$. Now we consider the case where at certain moments $\theta_{1}, \theta_{2}, \theta_{3}, \cdot \cdot \cdot$ such that $0 < \theta_{1} < \theta_{2}<, \cdot \cdot \cdot, \theta_{m} < \theta_{m + 1} = T, \lim_{k \to \infty}\theta_{k} = \infty,$ impulse effect act on the population ``momentarily", so that the population $p(\theta)$ varies by jump. And $p(\theta_{k}^{+})$ and $p(\theta_{k}^{-})$ present the population of the species before and after the impulsive effect at the time $\theta_{k}.$ \\
\indent Consider a set of all ld continuous function $\mathcal{C}(\mathcal{I}_{T}, \mathbb{R})$ from $\mathcal{I}_{T}$ to $\mathbb{R}.$ Set $\mathcal{I}_{0} = [0, \theta_{1}],$ and $\mathcal{I}_{k} = [\theta_{k}, \theta_{k + 1}]$ for each $k = 1, 2, 3, \cdot \cdot \cdot, m.$\\
 Consider 
\begin{center}
$\mathcal{PC}(\mathcal{I}_{T}, \mathbb{R}) = \{p: \mathcal{I}_{k} \to \mathbb{R}, p \in \mathcal{C}(\mathcal{I}_{T}, \mathcal{R}),$ $p(\theta_{k}^{+})$ and $p(\theta_{k}^{-})$ exist with $p(\theta_{k}^{-}) = p(\theta_{k}), k = 1, 2, 3, \cdot \cdot \cdot, m$\}, 
\end{center}
and
\begin{center}
 $\mathcal{PC}^{1}(\mathcal{I}_{T}, \mathbb{R}) = \{p: \mathcal{I}_{k} \to \mathbb{R}, p  \in \mathcal{C}^{1}(\mathcal{I}_{T}, \mathbb{R}), k = 1, 2, 3, \cdot \cdot \cdot, m$\},
 \end{center} where $\mathcal{C}^{1}(\mathcal{I}_{T}, \mathbb{R})$ is the set of all function from $\mathcal{I}_{k}$ to $\mathbb{R}.$ which is ld continuously nabla differentiable function. $\mathcal{PC}(\mathcal{I}_{T}, \mathbb{R})$ is a Banach space coupled with the norm \begin{center}$||p||_{\mathcal{PC}} = \sup_{\theta \in \mathcal{I}_{T}}|p(\theta)|.$\end{center}.
\begin{defn}\label{d12}
A function $p\in \mathcal{PC}^{1}(\mathcal{I}_{T}, \mathbb{R})$ is called a solution of the Equation (\ref{eq1}), if $p$ satisfies the Equation (\ref{eq1}) on $\mathcal{I}_{T}$ involving the condition 
$p(\theta_{k}^{+}) - p(\theta_{k}^{-}) = \mathscr{I}_{k}(\theta_{k}, p(\theta_{k}^{-})$ and $p(0) = \phi(T)$.
\end{defn}
\begin{lem}\label{l2}
Consider a ld continuous function $\mathscr{H}: \mathcal{I}_{T} \to \mathbb{R}.$ Then the solution of the problem for $k = 1, 2, 3,\dots, m,$ is
\begin{equation}\label{eq6}
\begin{cases}
^{C}\mathscr{D}^{w}p(\theta) = \mathscr{H}(\theta), &\theta \in \mathcal{I}_{T}, \theta \neq \theta_{k}\\
p(\theta_{k}^{+}) - p(\theta_{k}^{-}) = \mathscr{I}_{k}\big(\theta_{k}, p(\theta_{k}^{-}\big), & k = 1, 2, 3, \cdot \cdot \cdot, m\\
p(0) = \phi(p),
\end{cases}
\end{equation}
specified by the integral equation
\begin{equation}\label{eq7}
p(\theta)
\begin{cases}
\phi(p) + \frac{1}{\Gamma(w)}\int_{0}^{\theta}\big(\theta - \alpha(\zeta)\big)^{w - 1}\mathscr{H}(\zeta)\nabla \zeta, & \theta \in \mathcal{I}_{0}\\
\phi(p) + \frac{1}{\Gamma(w)}\sum_{i = 1}^{k}\int_{\theta_{i - 1}}^{\theta_{i}}\big(\theta_{i} - \alpha(\zeta)\big)^{w -1}\mathscr{H}(\zeta)\nabla \zeta +\\ \frac{1}{\Gamma(w)}\int_{\theta_{k}}^{\theta}\big(\theta - \alpha(\zeta)\big)^{w -1}\mathscr{H}(\zeta)\nabla \zeta + \sum_{i = 1}^{k}\mathscr{I}_{i}\big(\theta_{i}, p(\theta_{i}^{-})\big), & \theta \in \mathcal{I}_{k}.
\end{cases}
\end{equation}
\end{lem}
\begin{proof}
If $\theta \in \mathcal{I}_{0},$ then the solution of Equation (\ref{eq6}) is given by 
\begin{align}\label{eq8}
p(\theta) = \phi(p) + \frac{1}{\Gamma(w)}\int_{0}^{\theta}\big(\theta - \alpha(\zeta)\big)^{w - 1}\mathscr{H}(\zeta)\nabla \zeta.
\end{align}
For $\theta \in \mathcal{I}_{1},$ the problem
$$\begin{cases}
^{C}\mathscr{D}^{w}p(\theta) = \mathscr{H}(\theta),\\
p(\theta_{1}^{+}) - p(\theta_{1}^{-}) = \mathscr{I}_{1}\big(\theta_{1}, p(\theta_{1}^{-})\big),
\end{cases}$$
hold the solution
\begin{align}\label{eq9}
p(\theta) = p(\theta_{1}^{+}) + \frac{1}{\Gamma(w)}\int_{\theta_{1}}^{\theta}\big(\theta - \alpha(\zeta)\big)^{w -1}\mathscr{H}(\zeta)\nabla \zeta.
\end{align} 
 So from
\begin{align}\label{eq10}
p(\theta_{1}^{+}) - p(\theta_{1}^{-}) = \mathscr{I}_{1}\big(\theta_{1}, p(\theta)\big),
\end{align}
we get
\begin{align*}
p(\theta) = p(\theta_{1}^{-}) + \mathscr{I}_{1}\big(\theta_{1}, p(\theta_{1}^{-}) + \frac{1}{\Gamma(w)}\int_{0}^{\theta}\big(\theta - \alpha(\zeta)\big)^{w -1}\mathscr{H}(\zeta)\nabla \zeta.
\end{align*} which follows that 
\begin{align*}
p(\theta) = &\phi(p) + \mathscr{I}_{1}\big(\theta_{1}, p(\theta_{1}^{-})\big) + \frac{1}{\Gamma(w)}\int_{\theta_{1}}^{\theta}\big(\theta - \alpha(\zeta)\big)^{w - 1}\mathscr{H}(\zeta)\nabla \zeta \\
& + \frac{1}{\Gamma(w)}\int_{0}^{\theta}\big(\theta - \alpha(\zeta)\big)^{w - 1}\mathscr{H}(\zeta)\nabla \zeta,    \theta \in \mathcal{I}_{1}.
 \end{align*}
 Generalizing in this way, by using the principle of mathematical induction, for $\theta \in \mathcal{I}_{k}, k = 1, 2, 3, \cdot\cdot\cdot, m$ we conclude that
 \begin{align*}
 p(\theta) = &\phi(p) + \frac{1}{\Gamma(w)}\int_{\theta_{k}}^{\theta}\big(\theta - \alpha(\zeta)\big)^{w - 1}\mathscr{H}(\zeta)\nabla \zeta + \sum_{i = 1}^{k}\frac{1}{\Gamma(w)}\int_{\theta_{i - 1}}^{\theta_{i}}\big(\theta - \alpha(\zeta)\big)^{w - 1}\mathscr{H}(\zeta)\nabla \zeta\\ +&\sum_{i = 1}^{k}\mathscr{I}_{i}(\theta_{i}, p(\theta_{i}),       k = 1, 2, 3, \cdot \cdot \cdot, m.
 \end{align*}
To establish the existence and uniqueness of the impulsive dynamic Equation (\ref{eq1}), we need to assume the following conditions:\\
$(\mathcal{A}_{1})$ The mapping $\mathcal{L}: \mathscr{I} \times \mathbb{R} \times \mathbb{R} \to \mathbb{R}$ is a ld continuous and there must have constants $\mathscr{K} > 0$ and $0 < \mathscr{G} < 1$ which satisfies
$$|\mathcal{L}(\theta, \zeta_{1}, \zeta_{2}) - \mathcal{L}(\theta, \eta_{1}, \eta_{2})| \le \mathscr{K}|\zeta_{1} - \eta_{1}| + \mathscr{G}|\zeta_{2} - \eta_{2}|,
\forall \theta \in \mathcal{J}, \zeta_{i}, \eta_{i} \in \mathbb{R}.$$ 
$(\mathcal{A}_{2}).$ There exist constants $\mathscr{A}> 0, \mathscr{F} > 0$ and $0 < \mathscr{E} < 1,$ such that
$$|\mathcal{L}(\theta, \zeta, \eta)| \le \mathscr{A} + \mathscr{F}|\zeta| + \mathscr{E}|\eta|,  \forall \zeta, \eta \in \mathbb{R}.$$
$(\mathcal{A}_{3}).$ The function $\mathscr{I}_{k}(\theta, p)$ is continuous for all $k = 1, 2, 3, \cdot\cdot\cdot, m$ and satisfies the following: \\
\indent  There exists a positive constant $\mathscr{M}_{k}$ for $k = 1, 2, 3, \cdot\cdot\cdot, m$ such that
\begin{align*}
|\mathscr{I}_{k}(\theta, p)| \le \mathscr{M}_{k}, \forall \theta \in \mathcal{I}_{k}, p \in \mathbb{R}.
\end{align*} 
 \indent And there exists a positive constant $\mathscr{L}_{k},$ for $k = 1, 2, 3,\cdot\cdot\cdot, m$ such that
\begin{align*}
|\mathscr{I}_{k}(\theta, p) - \mathscr{I}_{k}(\theta, h)| \le \mathscr{L}_{k}|p - h|, \forall \theta \in \mathcal{I}_{k}, p, h \in \mathbb{R}.
\end{align*} 
$(\mathcal{A}_{4})$: There exists a non negative increasing function $\mu: \mathbb{R^{+}} \to \mathbb{R^{+}}$ such that
\begin{center}
$|\phi(\theta)| \le \mu(|\theta|)$ for every $\theta \in \mathcal{I}_{T},$ 
\end{center}
and there exists a positive constant $\mathscr{H}$ such that
\begin{center}
$|\phi(\theta) - \phi(\zeta)| \le \mathscr{H}|\theta - \zeta|$ for $\theta, \zeta \in \mathcal{I}_{T}.$
\end{center}
\end{proof}
\begin{thm}\label{th2}
If all the conditions $(\mathcal{A}_{1})$ - $(\mathcal{A}_{4})$ and 
\begin{center}$\sum_{i = 1}^{m}\mathscr{L}_{i} + \mathscr{H} + \frac{\mathscr{K}T^{w}(m + 1)}{(1 - \mathscr{G})\Gamma(w + 1)} < 1$ are hold,
\end{center} then the Equation (\ref{eq1}) has a unique solution on the interval $\mathcal{I}_{T}.$
\end{thm}
\begin{proof}
Let $^{C}\mathcal{D}^{w}p(\theta) = h(\theta).$ Consider a set $\Pi \subseteq \mathcal{PC}(\mathcal{I}_{k}, \mathbb{R}),$ such that
\begin{center}  $\Pi = \{p \in \mathcal{PC}^{1}(\mathcal{I}_{k}, \mathbb{R}): ||p||_{\mathcal{PC}} \le \sigma\},$
\end{center}
and an operator $\mathcal{X}: \Pi \to \Pi$ such that 
\begin{align*}
(\mathcal{X}p)(\theta) =  \phi(p) + \frac{1}{\Gamma(w)}\int_{0}^{\theta}\big(\theta - \alpha(\zeta)\big)^{w - 1}\mathcal{L}(\theta, p(\theta), ^{C}\mathscr{D}^{w}p(\theta))\nabla \zeta,
\end{align*} for $\theta \in \mathcal{I}_{0}.$
And 
\begin{align*}
(\mathcal{X}p)(\theta) =&  \phi(p) + \frac{1}{\Gamma(w)}\sum_{i = 1}^{k}\int_{\theta_{i - 1}}^{\theta_{i}}\big(\theta - \alpha(\zeta)\big)^{w - 1}\mathcal{L}(\theta, p(\theta), h(\theta))\nabla \zeta + \sum_{i = 1}^{k}\mathscr{I}_{i}(\theta_{i}, p(\theta_{i}^{-}))\\ &+ \frac{1}{\Gamma(w)}\int_{\theta_{k}}^{\theta}\big(\theta - \alpha(\zeta)\big)^{w - 1}\mathcal{L}(\theta, p(\theta), ^{C}\mathscr{D}^{w}p(\theta))\nabla \zeta,
\end{align*} for $\theta \in \mathcal{I}_{k}, k = 1, 2, 3, . . ., m.$\\

Case 1: When $\theta \in \mathcal{I}_{k},$ then for any $p \in \Pi,$ we get
\begin{align*}
|(\mathcal{X}p)(\theta)| =& |\phi(p)| + \big|\frac{1}{\Gamma(w)}\sum_{i = 1}^{k}\int_{\theta_{i - 1}}^{\theta_{i}}\big(\theta - \alpha(\zeta)\big)^{w - 1}h(\zeta)\nabla \zeta\big|\\& + \big|\sum_{i = 1}^{k}\mathscr{I}_{k}(\theta_{i}, p(\theta_{i}^{-}))\big| + \big|\frac{1}{\Gamma(w)}\int_{\theta_{k}}^{\theta}\big(\theta - \alpha(\zeta)\big)^{w - 1}h(\zeta)\nabla \zeta\big|,
\end{align*}
where $h \in \mathcal{PC}(\mathcal{I}_{T}, \mathbb{R}), \theta \in \mathcal{I}_{T},$ then from the Equation (\ref{eq1}) we get
$h = \mathcal{L}(\theta, p, h).$ Hence
\begin{align}
|h| = &|\mathcal{L}(\theta, p, h)|\nonumber\\
& \le \mathscr{A} + \mathscr{F}|p(\theta)| + \mathscr{E}|h(\theta)|\nonumber\\
& \le \frac{\mathscr{A}+ \mathscr{F}\sigma}{1 - \mathscr{E}}\label{eq11}. 
\end{align}Again taking the norm of $\mathcal{PC}(\mathcal{I}_{T}, \mathbb{R}),$ we get
\begin{center}
$||h||_{\mathcal{PC}} \le \frac{\rho + \mathscr{F}||p||_{\mathcal{PC}}}{1 - \mathscr{E}},$ where $||\mathscr{A}||_{\mathcal{PC}} = \rho.$
\end{center}
 Now by using the Proposition \ref{p01} along with the condition of Case 1, we get
\begin{align}
||\mathcal{X}p||_{\mathcal{PC}}& = \sup_{\theta \in \mathcal{J}_{T}}|\mathcal{X}p(\theta)|\nonumber\\
& \le \mu|p| +\sum_{i = 1}^{m}\mathscr{M}_{i} + \frac{\mathscr{A}+ \mathscr{F}|p|}{(1 - \mathscr{E})\Gamma(w)}\Big[\sum_{i = 1}^{m}\int_{\theta_{i - 1}}^{\theta_{i}}(\theta - \zeta)^{w - 1} d\zeta + \int_{\theta_{k}}^{\theta}\big(\theta - \zeta\big)^{w - 1}d\zeta\Big]\nonumber\\
& \le \mu\sigma +\sum_{i = 1}^{m}\mathscr{M}_{i} +  \frac{T^{w}(\rho + \mathscr{F}\sigma)(m + 1)}{\Gamma(w + 1)(1 - \mathscr{E})}\nonumber\\
& \le \sigma\label{eq13},
\end{align} where 
\begin{align*}
\sigma = \frac{\sum_{i = 1}^{m}\mathscr{M}_{i} + \frac{(m + 1)T^{w}\rho}{\Gamma(w + 1)(1 - \mathscr{E})}}{1 - \mu + \frac{(m + 1)T^{w}\mathscr{F}}{\Gamma(w + 1)(1 - \mathscr{E})}}.
\end{align*}
Case 2: if $\theta \in \mathcal{I}_{0},$ then similarly we can show that
\begin{align}
||\mathcal{X}p||_{\mathcal{PC}}& \le \mu\sigma + \frac{T^{w}\big(\rho + \mathscr{F}\sigma\big)}{(1 - \mathscr{E})\Gamma(w + 1)}\nonumber\\
& \le \sigma.\label{eq14}
\end{align}
Thus from the Equation (\ref{eq13}) and Equation (\ref{eq14}), we get
$||\mathcal{X}p||_{\mathcal{PC}} \le \sigma.$\\
Again for $p, q \in \Pi,$ we get 
\begin{align}
||\mathcal{X}p - \mathcal{X}q||_{\mathcal{PC}} =& \sup_{\theta \in \mathcal{I}_{k}}|(\mathcal{X}p)(\theta) - (\mathcal{X}q)(\theta)|\nonumber\\
&\le\sum_{i = 1}^{k}\big|\mathscr{I}_{i}(\theta_{i}, p(\theta_{i}^{-})) - \mathscr{I}_{i}(\theta_{i}, q(\theta_{i}^{-}))\big| + \frac{1}{\Gamma(w)}\big|\int_{\theta_{k}}^{\theta}\big(\theta - \alpha(\zeta)\big)^{w - 1}(h(\zeta) - g(\zeta))\nabla \zeta\big| \nonumber\\& + \frac{1}{\Gamma(w)}\big|\sum_{i = 1}^{k}\int_{\theta_{i - 1}}^{\theta_{i}}\big(\theta - \alpha(\zeta)\big)^{w - 1}(h(\zeta) - g(\zeta))\nabla \zeta\big| + |\phi(p) - \phi(q)|\label{eq15},
\end{align}
where $g \in \mathcal{PC}(\mathcal{J}_{T}, \mathbb{R}),$ given by $g(\theta) = \mathcal{L}(\theta, q(\theta), g(\theta)),$ and for $\theta \in \mathcal{I}_{T},$
\begin{align*}
|h(\theta) - g(\theta)| = &\big|\mathcal{L}(\theta, p(\theta), h(\theta)) - \mathcal{L}(\theta, q(\theta), g(\theta))\big|\\ \le &\mathscr{K}|p(\theta) - q(\theta)| + \mathscr{G}|h(\theta) - g(\theta)|\\
\le & \frac{\mathscr{K}|p(\theta) - q(\theta)|}{1 - \mathscr{G}}.
\end{align*}
Taking the norm of $\mathcal{PC}(\mathcal{I}_{T}, \mathbb{R}),$ we get
\begin{align}\label{eq16}
||h - g||_{\mathcal{PC}} \le \frac{\mathscr{K}||p - q||_{\mathcal{PC}}}{1 - \mathscr{G}}.
\end{align}
Thus, from the Equation (\ref{eq15}), we obtain 
\begin{align}
||\mathcal{X}p - \mathcal{X}q||_{\mathcal{PC}} & \le \sum_{i = 1}^{m}\mathscr{L}_{i}|p(\theta_{i}^{-}) - q(\theta_{i}^{-})| + \frac{\mathscr{K}|p(\zeta) - q(\zeta)|}{(1 - \mathscr{G})\Gamma(w)}\int_{\theta_{k}}^{\theta}(\theta - \zeta\big)^{w - 1}d\zeta\nonumber\\& + \frac{\mathscr{K}|p(\zeta) - q(\zeta)|}{(1 - \mathscr{G})\Gamma(w)}\sum_{i = 1}^{m}\int_{\theta_{i - 1}}^{\theta_{i}}\big(\theta - \zeta\big)^{w - 1}d\zeta + \mathscr{H}|p - q|\nonumber\\
& \le ||p - q||_{\mathcal{PC}}\sum_{i = 1}^{m}\mathscr{L}_{i} + \frac{\mathscr{K}T^{w}||p - q||_{\mathcal{PC}}}{(1 - \mathscr{G})\Gamma(w + 1)} + \frac{m\mathscr{K}T^{w}||p - q||_{\mathcal{PC}}}{(1 - \mathscr{G})\Gamma(w + 1)} + \mathscr{H}||p - q||_{\mathcal{PC}}\nonumber\\
&\le \Big(\sum_{i = 1}^{m}\mathscr{L}_{i} + \frac{\mathscr{K}T^{w}(m + 1)}{(1 - \mathscr{G})\Gamma(w + 1)} + \mathscr{H}\Big)||p - q|_{\mathcal{PC}}\label{eq17}.
\end{align}
Similarly for $\theta \in \mathcal{I}_{0},$ we get
\begin{align}\label{eq18}
||\mathcal{X}p - \mathcal{X}q||_{\mathcal{PC}} \le \Big(\mathscr{H} + \frac{\mathscr{K}T^{w}}{(1 - \mathscr{G})\Gamma(w + 1)}\Big)||p - q||_{\mathcal{PC}}.
\end{align}
From the Equation (\ref{eq17}) and Equation (\ref{eq18}), we have 
\begin{align*}
||\mathcal{X}p - \mathcal{X}q||_{\mathcal{PC}} \le \mathscr{U}||p - q||_{\mathcal{PC}},
\end{align*}where
\begin{center}$\mathscr{U} = \sum_{i = 1}^{m}\mathscr{L}_{i} + \frac{\mathscr{K}T^{w}(m + 1)}{(1 - \mathscr{G})\Gamma(w + 1)} + \mathscr{H}.$\end{center}
Since $\mathscr{U} < 1,$ so the operator $\mathcal{X} : \Pi \to \Pi$ is a contraction operator hence it has a fixed point by Banach contraction theorem, which is the solution of the Equation (\ref{eq1}).
\end{proof}
\indent We use the Schaefer's fixed point theorem to show the sufficient condition of the existence of the solution for the Equation (\ref{eq1}).
\begin{thm}\label{th3}
If the assumptions $(\mathcal{A}_{1}) - (\mathcal{A}_{4})$ are satisfied and there exists a positive constant $\beta$ such that 
\begin{equation}\label{eq19}
\mu\beta + \sum_{i = 1}^{m}\mathscr{M}_{i} + \frac{(m + 1)T^{w}(\mathscr{A} + \mathscr{F}\beta)}{\Gamma(w + 1)(1 - \mathscr{E})} < \beta,
\end{equation}
then Equation (\ref{eq1}) has at least one solution on $\mathcal{I}_{T}.$
\end{thm}
\begin{proof}
The proof of the theorem is given in the following steps:\\
Step 1: The operator $\mathcal{X}: \Pi \to \Pi$ is continuous; Let $\{p_{n}\}$ be a sequence of $\mathcal{PC}(\mathcal{I}_{T}, \mathbb{R})$ such that $p_{n} \to p,$ then for each $\theta \in \mathcal{I}_{k}$ for $k = 1, 2, 3, \cdot\cdot\cdot, m$
\begin{align}
||\mathcal{X}p_{n} - \mathcal{X}q||_{\mathcal{PC}}& = \sup_{\theta \in \mathcal{I}_{k}}|(\mathcal{X}p_{n})(\theta) - (\mathcal{X}q)(\theta)|\nonumber\\
& \le \sum_{i = 1}^{m}\Big|\mathscr{I}_{i}(\theta_{i}, p_{n}(\theta_{i}^{-})) - \mathscr{I}_{i}(\theta_{i}, p(\theta_{i}^{-}))\Big| + \Big|\frac{1}{\Gamma(w)}\int_{\theta_{k}}^{\theta}\big(\theta - \zeta)\big)^{w - 1}(h_{n}(\zeta) - h(\zeta))d\zeta\Big|\nonumber\\& + \frac{1}{\Gamma(w)}\Big|\sum_{i = 1}^{m}\int_{\theta_{i - 1}}^{\theta_{i}}\big(\theta_{i} - \zeta\big)^{w - 1}(h_{n}(\zeta) - h(\zeta)d\zeta\Big| + |\phi(p_{n}) - \phi(p)|\label{eq20},
\end{align}
where $h_{n} \in \Pi,$ such that $h_{n} = \mathcal{L}(\theta, p_{n}, h_{n})$ then from the Equation (\ref{eq01}), we get 
\begin{align}
|h_{n} - h|& = |\mathcal{L}(\theta, p_{n}, h_{n}) - \mathcal{L}(\theta, p, h)|\nonumber\\
&\le \mathscr{K}|p_{n} - p| + \mathscr{G}|h_{n} - h|\nonumber\\
&\le \frac{\mathscr{K}|p_{n} - p|}{1 - \mathscr{G}}\label{eq21}.
\end{align}
If we take the norm of $\mathcal{PC}(\mathcal{I}_{T}, \mathbb{R}),$ then we get
\begin{align}\label{eq22}
||h_{n} - h||_{\mathcal{PC}} \le \frac{\mathscr{K}}{1 - \mathscr{G}}||p_{n} - p||_{\mathcal{PC}}.
\end{align}
Now using the Equation (\ref{eq22}), we have from the Equation (\ref{eq20})
\begin{align*}
||\mathcal{X}p_{n} - \mathcal{X}q||_{\mathcal{PC}} &\le ||p_{n} - p||_{\mathcal{PC}}\Big(\sum_{i = 1}^{m}\mathscr{L}_{i} + \frac{\mathscr{K}T^{w}(m + 1)}{(1 - \mathscr{G})\Gamma(w + 1)} + \mathscr{H}\Big).
\end{align*}
If $p_{n} \to p$ as $n \to \infty,$ then $||\mathcal{X}p_{n} - \mathcal{X}q||_{\mathcal{PC}} \to 0.$ Hence the operator is continuous.\\
Similarly for $\theta \in \mathcal{I}_{0},$ we can show that $||\mathcal{X}p_{n} - \mathcal{X}q||_{\mathcal{PC}} \to 0,$ as $n \to \infty.$ \\
Step 2: The operator $\mathcal{X}$ conveys equicontinuous set of $\mathcal{PC}(\mathcal{I}_{T}, \mathbb{R})$ to equicontinuous set on $\mathcal{PC}(\mathcal{I}_{T}, \mathbb{R}).$ Let $x_{1}, x_{2} \in \mathcal{I}_{k}, k =1, 2, 3, \cdot\cdot\cdot, m,$ then we get
\begin{align*}
||\mathcal{X}p(x_{2}) - \mathcal{X}p(x_{1})||_{\mathcal{PC}}& = \sup_{\theta \in \mathcal{I}_{k}}\big|\mathcal{X}p(x_{2}) - \mathcal{X}q(x_{1})\big|\\
\le& \big|\frac{1}{\Gamma(w)}\int_{\theta_{k}}^{x_{1}}\big((x_{2} - \alpha(\zeta))^{w - 1} - (x_{1} - \alpha(\zeta)^{w - 1})\big)h(\zeta)\nabla \zeta\big|\\& + \big|\frac{1}{\Gamma(w)}\int_{x_{1}}^{x_{2}}(x_{2} - \alpha(\zeta))^{w - 1}h(\zeta)\nabla \zeta\big| + \sum_{0 < \theta_{k} < x_{2} - x_{1}}\big|\mathscr{I}_{\theta_{k}}(\theta_{k}, p(\theta_{k}^{-}))\big|\\
< &\Big|\frac{1}{\Gamma(w)}\int_{\theta_{k}}^{x_{1}}\big((x_{2} - \zeta)^{w - 1} - (x_{1} - \zeta)^{w - 1}\big)h(\zeta)d\zeta\Big| + \Big|\frac{1}{\Gamma(w)}\int_{x_{1}}^{x_{2}}(x_{2} - \zeta)^{w - 1}h(\zeta)d\zeta\Big|\\& + \sum_{0 < \theta_{k} < x_{2} - x_{1}}\big|\mathscr{I}_{\theta_{k}}(\theta_{k}, p(\theta_{k}^{-}))\big|\\
\le & \frac{\mathscr{A} + \mathscr{F}\sigma}{(1 - \mathscr{E})\Gamma(w)}\Big(\Big|\int_{\theta_{k}}^{x_{1}}\big((x_{2} - \zeta)^{w - 1} - (x_{1} - \zeta)^{w - 1}\big)d\zeta\Big| + \Big|\frac{1}{\Gamma(w)}\int_{x_{1}}^{x_{2}}(x_{2} - \zeta)^{w - 1}d\zeta\Big|\Big) \\&+ \sum_{0 < \theta_{k} < x_{2} - x_{1}}\big|\mathscr{I}_{\theta_{k}}(\theta_{k}, p(\theta_{k}^{-}))\big|.
\end{align*}
Since $(x - \zeta)^{w - 1}$ is continuous, if $x_{1} \to x_{2},$ then we get
\begin{align*}
||\mathcal{X}p(x_{2}) - \mathcal{X}p(x_{1})||_{\mathcal{PC}} \to 0.
\end{align*}
Since the result at $x_{1}, x_{2} \in \mathcal{I}_{0}$ is similar, so the proof is omitted.\\
Step 3: The operators $\mathcal{X}$ maps bounded sets in $\mathcal{PC}(\mathcal{I}_{T}, \mathbb{R}).$ If we consider the set $\Pi$ as same as in Theorem \ref{th2}, then we obtain from the Equation (\ref{eq13}) that the set $\Pi$ is closed, convex and bounded of $\mathcal{PC}(\mathcal{I}_{T}, \mathbb{R}).$ Then for any $p \in \Pi,$ we get
$$||\mathcal{X}(p)|| \le \sigma,$$ which indicate the boundedness condition of the operator $\mathcal{X}.$
As a consequences of the Step 1, Step 2 and Step 3 together with the theorem of Arzela-Ascoli, we arrived that the mapping $\mathcal{X}$ is continuous completely.\\
Step 4: For any $\lambda \in (0, 1),$ the set 
$\mathcal{K} = \{p \in \mathcal{PC}(\mathcal{I}_{k}, \mathbb{R}): p = \lambda\mathcal{X}(p), 0 < \lambda < 1\}$ is bounded,  for $\theta \in \mathcal{I}_{k}, k = 1, 2, 3, \cdot\cdot\cdot, m,$ we have 
\begin{align*}
|p(\theta)|& = |\lambda\mathcal{X}(p)\theta|\\
&= \Big|\lambda\Big(\phi(p) + \frac{1}{\Gamma(w)}\sum_{i = 1}^{k}\int_{\theta_{i - 1}}^{\theta_{i}}\big(\theta_{i} - \alpha(\zeta)\big)^{w - 1}h(\zeta)\nabla\zeta\\& + \frac{1}{\Gamma(w)}\int_{\theta_{k}}^{\theta}\big(\theta - \alpha(\zeta)\big)^{w - 1}h(\zeta)\nabla \zeta + \sum_{i = 1}^{k}\mathscr{I}_{i}(\theta_{i}, p(\theta_{i}^{-}))\Big)\Big|\\
& \le \mu||p||_{\mathcal{PC}} + \sum_{i = 1}^{n}\mathscr{M}_{i} + \frac{(\mathscr{A} + \mathscr{F}||p||_{\mathcal{PC}})T^{w}(m + 1)}{\Gamma(w + 1)(1 - \mathscr{E})}.
\end{align*}
Thus,
\begin{align*}
\frac{||p||_{\mathcal{PC}}}{\mu||p||_{\mathcal{PC}} + \sum_{i = 1}^{n}\mathscr{M}_{i} + \frac{(\mathscr{A} + \mathscr{F}||p||_{\mathcal{PC}})T^{w}(m + 1)}{(1 - \mathscr{E})\Gamma(w + 1)}} \le 1.
\end{align*} 
From the Equation (\ref{eq19}), we get a positive constant $\beta$ such that $||p||_{\mathcal{PC}} \neq \beta.$ Consider a set $\Psi = \{p \in \mathcal{PC}(\mathcal{I}_{T}, \mathbb{R}): ||p||_{\mathcal{PC}} < \beta\}$. Then the operator $\mathcal{X}: \overline{\Psi} \to \mathcal{PC}(\mathcal{I}_{T}, \mathbb{R})$ is continuous and completely continuous. So there is no $p \in \partial(\Psi)$ such that $p = \lambda\mathcal{X}(p), \lambda \in (0, 1).$ Thus, from the Scheafer's fixed point theorem we get that the operator $\mathcal{X}$ has a fixed point, which is the solution of the Equation (\ref{eq1}).\\
\indent The result for $\theta \in \mathcal{I}_{0}$ is almost same, so it is omitted.
\end{proof}

\section{Example}
\begin{exm}
Consider a impulsive fractional dynamic equation involving a non-local initial condition 
\begin{equation}\label{eq23}
\begin{cases} 
^{C}\mathcal{D}^{\frac{1}{2}}p(\theta) = \frac{e^{-3\theta}\big(2 + |p(\theta)| + |^{C}\mathcal{D}^{w}p(\theta)|\big)}{35e^{2\theta}\big(1 + |p(\theta)| + |^{C}\mathcal{D}^{w}p(\theta)|\big)}, &\theta = [0, 1]_{T}, \theta \neq \frac{1}{3}\\
p(\theta_{1}^{+}) - p(\theta_{1}^{-}) = \frac{1 + \theta_{1}^{-}e^{p(\theta_{1}^{-})}}{10}, &\theta_{1} = \frac{1}{3}\\
p(0) = \frac{1 + e^{p}}{5}. 
\end{cases}
\end{equation}
We set 
\begin{align*}
\mathcal{L}(\theta, p, q)& = \frac{e^{-3\theta}\big(2 + |p(\theta)| + |q(\theta)|)}{35e^{2\theta}\big(1 + |p(\theta)| + |q(\theta)|)}\\
\mathscr{I}_{1}(\theta, p)& = \frac{1 + \theta e^{p(\theta)}}{10}, \theta = \theta_{1}, p \in \mathbb{R}\\
\phi(p) &= \frac{1 + e^{p}}{5},
\end{align*}
then for all $\theta \in [0, 1]_{T}$ and  $p, q, h, g \in \mathbb{R},$ we get
\begin{align*}
|\mathcal{L}(\theta, p, q)| \le \frac{2 + |p(\theta)| + |q(\theta)|}{35e^{2\theta}},
\end{align*}and 
\begin{align*}
|\mathcal{L}(\theta, p, q) - \mathcal{L}(\theta, h, g)| \le \frac{1}{35e^{3}}|p - h| + \frac{1}{35e^{3}}|q - g|,\\
|\mathscr{I}_{1}(\theta, p) - \mathscr{I}_{1}(\theta, q)| \le \frac{1}{10}|p - h|, |\phi(p) - \phi(h)| \le \frac{1}{5}|p - h|, |\phi(p)| \le \frac{2}{5}.
\end{align*}
From here we get $\mathscr{K} = \frac{1}{35e^{3}}, \mathscr{G} = \frac{1}{35e^{3}}, \mathscr{L}_{1} =\frac{1}{10}, \mathscr{H} = \frac{2}{5},$ therefore  the conditions of $\mathcal{A}_{1} - \mathcal{A}_{4}$ are satisfied, thus for $m = 1$, we get
\begin{align*}
\mathscr{L}_{1} + \frac{\mathscr{K}T^{w}(p + 1)}{(1 - \mathscr{L})\Gamma(w + 1)} + \mathscr{H} \le& \frac{1}{10} + \frac{2}{5} + \frac{2\frac{1}{35e^{3}}}{(1 - \frac{1}{35e^{3}})\Gamma(\frac{1}{2} + 1)}\\
\le & \frac{1}{2} + \frac{2\frac{1}{35e^{3}}}{(1 - \frac{1}{35e^{3}})\Gamma(\frac{1}{2} + 1)}\\ \le &1.
\end{align*}
Thus, the conditions of the Theorem \ref{th2} are satisfied. Therefore we can say that the Equation (\ref{eq23}) has a solution. 
\end{exm}

\section{Conclusion}
In this manuscript we discussed the Caputo nabla derivative and Riemann-Liouville nabla derivative and compare both the operators in the time scale context. Also we discussed fractional dynamic equation by the Caputo nabla derivative involving instantaneous impulses with non local initial condition. The discussion of the stability of the solution of the Equation (\ref{eq1}) will be our future work. The theory of impulsive fractional dynamic equation has a potential application on the filed of mathematical analysis, moreover it has a wide application in physics, field of engineering, economics, etc.\\

\noindent {\bf Funding:} This study was not funded by any agencies. \\
{\bf Competing interests:}
The authors declare that they have no competing interests.\\
{\bf Authors’ contributions:}
Each of the authors contributed to each part of this work equally and read and approved the final version of the manuscript.\\
{\bf Ethical approval:}
This article does not contain any studies with human participants or animals performed by any of the authors.

\end{document}